\documentclass[12pt,leqno]{amsart}
\usepackage{tikz}
\usetikzlibrary{automata,positioning}
\usepackage{amssymb,amsthm,amsmath,latexsym}
\usepackage{graphicx} 
\graphicspath{{figures/}} 
\usepackage{booktabs} 
\usepackage[font=small,labelfont=bf]{caption} 
\usepackage{amsfonts, amsmath, amsthm, amssymb} 
\usepackage{wrapfig} 
\usepackage{multicol}

\newtheorem{theorem}{\sc Theorem}[section]

\newtheorem{lem}[theorem]{\sc Lemma}
\newtheorem{prop}[theorem]{\sc Proposition}

\newtheorem{ex}[theorem]{\sc Example}

 \newtheorem*{thmA}{Theorem A}
 \newtheorem*{thmB}{Theorem B}

\title{Self-similarity of some soluble relatively free groups}

\author{Adilson A. Berlatto}
\address{Instituto de Ciências Exatas e da Terra - ICET/CUA, Universidade Federal de Mato Grosso,
Pontal do Araguaia-MT, 78698-000 Brazil}
\email{(Berlatto) adilson.berlatto@gmail.com}

\author{Alex C. Dantas}
\address{Departamento de Matem\'atica, Universidade de Bras\'ilia,
Brasilia-DF, 70910-900 Brazil}
\email{(Dantas) alexcdan@gmail.com}

\author{Tulio M. G. Santos}
\address{Departamento de Matem\'atica, Centro de Tecnologia, Cidade Universit\'aria da Universidade Federal do Rio de Janeiro, Rio de Janeiro-RJ, 21941-909 Brazil}
\email{(Santos) tuliosantos.im@gmail.com}


\begin{document}
\maketitle

\begin{abstract}
In this paper we prove that a free nilpotent group of finite rank is transitive self-similar. In contrast, we prove that a free metabelian group of rank $r \geq 2$ is not transitive self-similar.
\end{abstract}

\section{Introduction}

A group $G$ is self-similar if the group has a faithful state-closed representation on an infinite regular one-rooted $m$-tree $\mathcal{T}_{m}$,  for some integer $m\geq 2$; in addition, if $G$ acts transitively on the first level of the tree, $G$ is said to be  transitive self-similar. Nekrashevych and Sidki \cite{NS} produced a method for construction of transitive self-similar groups via a virtual endomorphism; a group $G$ is transitive self-similar if and only if there exist a subgroup $H$ of index $m$ in $G$ and an endomorphism $f: H \rightarrow G$ such that the maximal $f$-invariant normal subgroup $K$ of $G$ contained in $H$ is trivial (in this case $f$ is called \emph{simple}).

The literature on self-similar groups is quite rich. They have been studied for
abelian groups \cite{BS}, finitely generated nilpotent groups \cite{AS}, affine linear groups \cite{SaS}, arithmetic groups \cite{K} and soluble groups \cite{BaSu}. Nekrashevych and Sidki \cite{NS} studied the structure of self-similar free abelian groups of finite rank in terms of their virtual endomorphisms. We use this approach to establish results on the self-similarity of finitely generated free nilpotent groups and finitely generated free metabelian groups. 

Following P. Hall's notations, we denote a finitely generated torsion-free nilpotent group of class $c$ by $\mathfrak{T}_c$. In \cite{AS} it was shown that if $G$ is a $\mathfrak{T}_2$-group and $H$ is a subgroup of finite index in $G$, then there exists a subgroup $K$ of finite index in $H$ which admits a simple surjective virtual endomorphism $f:K \to G$. A surjective virtual endomorphism is called \emph{recurrent} and, if it is simple, we say that  $G$ is a \emph{recurrent self-similar group}.  A group $G$ is called \textit{compressible} if any finite-index subgroup of $G$ contains a finite-index subgroup $K$ such that $K \simeq G$. In \cite{smith}, G. Smith showed that the free nilpotent group $N_{r,c}$ of class $c$ and finite rank $r$ is compressible.  We extend this result, observing that in the proof we can produce a transitive recurrent self-similar representation from each finite index subgroup of $N_{r,c}$: 

\begin{thmA}
The free nilpotent group $N_{r,c}$ of class $c$ and finite rank $r$ is recurrent transitive self-similar. Furthermore, $N_{r,c}$ is an automata group. 
\end{thmA} 

Recall that a group $G$ is said to be an \textit{automata group} if $G$ is generated by the states of an invertible finite state automata $\mathcal{A}$.

Denote the free metabelian group of rank $r$ by $\mathbb{M}_{r}$. In \cite{BS}, Brunner and Sidki showed that a free metabelian group of finite rank has a faithful finite state representation on the binary tree. We extend Theorem 1 of \cite{DS1} and use it to prove the following result.

\begin{thmB}
The free metabelian group $\mathbb{M}_{r}$ of rank $r \geq 2$ is not transitive self-similar. 
\end{thmB}

\section{Preliminaries}

\textbf{Self-similar groups and virtual endomorphisms.} Let $Y=\{1,...,m\}$ be a finite alphabet with $m\geq 2$ letters. The set of finite words $Y^*$ over $Y$ has a structure
of a rooted $m$-ary tree, denoted by $\mathcal{T}(Y)$ or $\mathcal{T}_m$. The incidence relation on $\mathcal{T}_m$ is given by: $(u,v)$ is an edge if and only if there exists a letter $y$ such that $v=uy$. The empty
word $\emptyset$ is the root of the tree and the level $i$ is the set of all  words of length $i$. 

The automorphism group $\mathcal{A}_{m}$, or $\mathcal{A}\left( Y\right) $,
of $\mathcal{T}_{m}$ is isomorphic to the restricted wreath product
recursively defined as $\mathcal{A}_{m}=\mathcal{A}_{m}\wr Perm(Y)$. An automorphism $\alpha $ of $%
\mathcal{T}_{m}$ has the form $\alpha =(\alpha _{1},...,\alpha _{m})\sigma
(\alpha )$, where the state $\alpha _{i}$ belongs to $\mathcal{A}_{m}$ and $\sigma :\mathcal{A}_{m}\rightarrow Perm(Y)$ is the permutational
representation of $\mathcal{A}_{m}$ on $Y$, the first level of the tree $%
\mathcal{T}_{m}$.
The action of $\alpha =(\alpha _{1},...,\alpha _{m})\sigma
(\alpha ) \in \mathcal{A}_m$ on a word $yu$ is given recursively by $(yu)^{\alpha}=y^{\sigma(\alpha)}u^{\alpha_y}$.

Given an element $\alpha$ that belongs to  $\mathcal{A}\left( Y\right) $, the set of automorphisms%
\begin{equation*}
Q(\alpha )=\{\alpha\}\cup Q(\alpha _{1}) \cup \dots \cup Q(\alpha _{m})
\end{equation*}%
is called the set of \textit{states} of $\alpha $ and this automorphism is
said to be \textit{finite-state} provided $Q(\alpha )$ is finite. A subgroup 
$G$ of $\mathcal{A}_{m}$ is \textit{state-closed} in the language of automata (or \textit{self-similar} in the language of dynamics) if $Q(\alpha )$ is a subset of $G$ for all $\alpha $ in $G$ and is \textit{transitive} if its action on the first level of the tree is transitive. 
A self-similar group  
which is finitely generated and finite-state is called an \textit{automata group}. 

Given a group $G$ and a virtual endomorphism $f:H \to G$ we produce a transitive state-closed action of $G$ on the regular rooted $m$-tree $\mathcal{T}_m$, where $m = [G : H]$. Let $T=\{t_1,...,t_{m}\}$ be a right transversal of $H$ in $G$ and consider $\sigma:G \to Perm(Y)$ given by $i^{\sigma(g)}=j$ if and only if $Ht_ig=Ht_j$. Note that $h_i=t_ig(t_{i^{\sigma(g)}})^{-1}\in H$. For each $g \in G$ we obtain 
$$(h_1,...,h_{m})\sigma(g) \in H \, {wr}_{T} \,G^{\sigma}.$$

Using the virtual endomorphism $f:H \to G$ we obtain a representation $\varphi:G \to \mathcal{A}_m$ defined recursively by

\begin{equation*}
\varphi :\text{ }g\mapsto \left(  h_{1} ^{f\varphi
},...,  h_{m} ^{f\varphi }\right) \,{\sigma(g) }.
\end{equation*}%

The kernel of the representation $\varphi$, called the $f$-core of $H$, that is the maximal  subgroup $K$ of $H$ which is normal in $G$ and $f$-invariant (in the sense $K^f \leqslant K$) \cite{NS}. If the $f$-core of $H$ is trivial then $G$ is a self-similar group and we say that $f$ is a \textit{simple} virtual endomorphism. \\

\noindent \textbf{Some results about nilpotent groups.}
We list below some facts about nilpotent groups. Such results can be found in \cite{Baumslag}, \cite{mhall} and \cite{phall}. 

The free nilpotent group $N_{r,c}$ of class $c$ and rank $r$ is isomorphic to the group  $\frac{F_{r}}{\gamma_{c+1}(F_{r})}$, where $F_r$ is the free group of rank $r$.
 The terms of the lower and the upper central series of $N_{r,c}$ coincide, that is, $\gamma_{i+1}(N_{r,c}) =  Z_{c-i}(N_{r,c})$, $\forall \; i  = 0,1, \ldots, c$. Also the quotients $\frac{N_{r,c}}{\gamma_i(N_{r,c})}$ are free nilpotent groups,  for all $ i =  2, \ldots, c$.

Let $G = \left\langle x_1, \ldots, x_r \right\rangle$ be a group. The commutators $c_j$ over $ \{x_1, \ldots, x_r\}$ and its weights $\omega(c_j)$ are inductively defined by:

\begin{enumerate}
    \item $c_i = x_i$, for $ i = 1, \ldots, r$, are the commutators of weight one;
    \item  if $c_i$ and $c_j$ are commutators, then $c_k = [c_i, c_j]$ is a commutator and $\omega(c_k) =  \omega(c_i)+\omega(c_j)$. 
\end{enumerate}

\noindent The basic commutators over $\{x_1, \ldots, x_r\}$ are useful for free nilpotent groups. They are defined inductively by: 

\begin{enumerate}
    \item weight one: $c_i = x_i$, for $ i = 1, \ldots, r$;
    \item weight $n \geq 2$:  $c_k=[c_i,c_j]$ where:
    \begin{enumerate}
        \item  $c_i$ and $c_j$ are basic commutators and $ \omega(c_i)+\omega(c_j)= n$;
        \item $i>j$ and if $c_i = [c_s, c_t]$, then $j \geq t$;
    \end{enumerate}
    \item basic commutators are ordered according their weights; in the case of  same weight, the order is arbitrary.
\end{enumerate}

\noindent We still define the weights $\omega_i(c)$ for $i = 1, \ldots, r$ by the rules $\omega_i(x_i) = 1$, $\omega_i(x_j) = 0$ if $i \neq j$ and recursively, $\omega_i([c_k, c_m]) = \omega_i(c_k) + \omega_i(c_m)$.

\noindent Consider $M_r(n)$ the number of basic commutators of weight $n$ over $\{x_1, \ldots, x_r\}$ and $M(n_1, \ldots, n_r)$ the number of basic commutators $c$ such that $\omega_i(c) = n_i$ for $i = 1, \ldots, r$, where $ n_1+ \cdots + n_r = n$ is a partition of $n$ into $r$ parts. Let $d = p_1^{m_1} \cdots p_k^{m_k}$ be a positive integer, where $p_1, \ldots, p_k$ are distinct primes and $m_i>0$, for all $i = 1, \ldots, k$. The {\it Mobius function} is defined by 
$$
\mu(d) =
\left\{
\begin{array}{ll}
1,           & \mbox{ if } d=1; \\
0,           & \mbox{ if } m_j >1, \;\; \mbox{ for some } j \in \{ 1, \ldots, k \}; \\
(-1)^k,      & \mbox{ if } d = p_1 p_2 \cdots p_k.
\end{array}
\right.
$$
Then we have the following results: \\
\begin{theorem}[Witt's formula]
\label{witt}
 With the above notation, \\
$$ M_r(n) = \frac{1}{n} \displaystyle\sum_{d|n} \mu(d) r^{\frac{n}{d}} \;\; \mbox{ and }$$
$$  M(n_1, n_2, \ldots, n_r) = 
\frac{1}{n}
\displaystyle\sum_{d|n_i}
\mu(d)
\frac{ \left( \frac{n}{d} \right) !} 
{\left( \frac{n_1}{d} \right) ! \left( \frac{n_2}{d} \right) ! 
\cdots \left( \frac{n_r}{d} \right) !}
.$$
\end{theorem}
\begin{theorem}
\label{basis} 
Let $F$ be a free group with basis $\{x_1, \ldots, x_r \}$. Then the basic commutators of weight $n$ over $x_1, \ldots, x_r $ form a basis for the free abelian group 
$\frac{\gamma_n(F)}{\gamma_{n+1}(F)}$.
\end{theorem}

\noindent We observe that if $G = N_{r,c}$, then 
$
\frac{\gamma_n(G)}{\gamma_{n+1}(G)} \simeq \frac{\gamma_n(F)}{\gamma_{n+1}(F)}$,  for $n = 1, \ldots, c$. Thus the rank of $\frac{\gamma_n(G)}{\gamma_{n+1}(G)}$ is $M_r(n)$. In particular, $\gamma_c(G)$ is free abelian of rank $M_r(c)$.

\noindent Let $G$ be a group. A subgroup $H$ of $G$ is said to be {\it isolated} in $G$ if the conditions $x \in G$ and $x^n \in H$, for some $n \geq 1$, imply $x \in H$. Following P. Hall, finitely generated torsion-free nilpotent groups are called $\mathfrak{T}$-groups. The following results concern to $\mathfrak{T}$-groups; the proofs can be found in \cite{Baumslag} and \cite{smith}.
\begin{theorem}
\label{smith} 
Let $G$ be a $\mathfrak{T}$-group such that $[G:HG']$ is finite. Then $[G:H]$ is finite.
\end{theorem}
\begin{theorem}
\label{baumslag}
Let $G$ be a $\mathfrak{T}$-group and $H$ an isolated subgroup of $G$. Then, for every prime $p$, we have
$$
\displaystyle\bigcap_{i \geq 1} G^{p^i} H = H.
$$
 \end{theorem}

\section{Free Nilpotent Groups}

We begin this section finding appropriate generators for isomorphic subgroups of $N_{r,c}$. First, consider $G = N_{r,c} = \left\langle g_1, \ldots, g_r \right\rangle$, $z_1, \ldots, z_r \in G'$ and $H = \left\langle g_1^{n_1}z_1, \ldots, g_r^{n_r}z_r \right\rangle$. As $H$ has finite index in $G$ modulo $G'$, it follows from Theorem \ref{smith} that $[G:H]$ is finite. Now, consider the map $\psi: G \longrightarrow H$ defined by $g_i^{\psi} = g_i^{n_i} z_i$, for all $i = 1, \ldots, r$. Then $\psi$ extends to an epimorphism. As $G$ and $H$ have the same Hirsch length, $\psi$ is also a monomorphism and we have that $H \simeq G$. In the opposite direction:

\begin{prop}\label{generators} 
Consider $G = N_{r,c}$ and $H$ a subgroup of $G$ which is isomorphic to $G$. Then there exists $g_1, \ldots, g_r \in G$, $z_1, \ldots, z_r \in G'$ and positive integers $n_1, \ldots, n_r$ such that $G = \left\langle g_1, \ldots, g_r \right\rangle$ and $H = \left\langle g_1^{n_1}z_1, \ldots, g_r^{n_r}z_r \right\rangle$.  
 \end{prop}
 
 \noindent {\it Proof:} As  $\left[ \frac{G}{G'}: \frac{HG'}{G'}\right]$ is finite, there exist $g_1, \ldots, g_r$ and positive integers $n_1, \ldots, n_r$ such that 
$\frac{G}{G'} = \left\langle G' g_1, \ldots, G' g_r\right\rangle$ and $\frac{HG'}{G'} = \left\langle G' g_1^{n_1}, \ldots, G' g_r^{n_r}\right\rangle$. So we have $G = \left\langle g_1, \ldots, g_r \right\rangle G' =  \left\langle g_1, \ldots, g_r \right\rangle$ and $HG' =  \left\langle g_1^{n_1}, \ldots, g_r^{n_r} \right\rangle G'$. We can choose $z_i \in G'$ such that $g_i^{n_i}z_i \in H$, for each $i = 1, \ldots, r$. Now, consider $K = \left\langle g_1^{n_1}z_1, \ldots, g_r^{n_r}z_r \right\rangle$. Then $K \leq H$, $KG' = HG'$ and
$$
H = HG' \cap H = KG' \cap H = (H\cap G') K,
$$
where the last equality follows from modular law.
But $H \cap G' = H \cap Z_{c-1}(G) = Z_{c-1}(H) = H'$ and $H= H'K$, that is, $H=K$. \hfill$\Box$ \\

Consider $n>1$ and $x_j$ a generator of a group $G =  \left\langle x_1, \ldots, x_r \right\rangle$. Let us count the number of times that $x_j$ appears in all basic commutators of weight $n$. We denote such number by $A_j(r,n)$. We will calculate
$$
A_j(r,n)  = \displaystyle{\sum_{\omega(c)=n}}\omega_j(c),$$
where the $c$'s on the subscript are the basic commutators over $x_1, \ldots, x_r$. Let $A_j(r,n,k)$ denote the number of basic commutators of weight $n$ where $x_j$ appears exactly $k$ times. By Theorem \ref{witt}, we have that  $A_j(r,n,k)$ doesn't depends on $j$, that is, 
$$
\begin{array}{ll}
A_j(r,n,k) & = 
\displaystyle{\sum_{ {n_1 + \cdots +n_r = n} \atop {n_j = k} }^{}}
M(n_1, n_2, \ldots, n_j, \ldots, n_r) \\
& = \displaystyle\sum_{k+n_2+ \ldots + n_r = n} M(k, n_2, \ldots, n_r). 
\end{array}
$$
Therefore, $A_j(r,n,k) = A(r,n,k)$ and
$$
A(r,n)  = A_j(r,n) = 
\displaystyle\sum_{k=1}^n k A(r,n,k).
$$
\begin{lem}\label{arn} 
Let $G$ be a group generated by $x_1, \ldots, x_r$. Then 
$$
A(r,n) = 
\displaystyle\sum_{\omega(c)=n} \omega_j(c) = 
\displaystyle\sum_{k=1}^n k \cdot
\left(
\displaystyle\sum_{k+n_2+ \ldots + n_r = n} M(k, n_2, \ldots, n_r)
\right),
$$
for all $\; j = 1, \ldots, r$, where $c$ runs over the set of all basic commutators of weight $n$ over $x_1, \ldots, x_r$.
 \end{lem}

Now we calculate the index of the isomorphic subgroups of $N_{r,c}$. 
\begin{theorem} 
\label{index} 
Consider $G = N_{r,c} =  \left\langle g_1, \ldots, g_r \right\rangle$ and $H = \left\langle g_1^{n_1}z_1, \ldots, g_r^{n_r}z_r \right\rangle$ a subgroup of $G$ with $H \simeq G$, where $n_1, \ldots, n_r$ are positive integers and $z_1, \ldots, z_r \in G'$. Then 
$$
[G:H] = (n_1n_2 \cdots n_r)^{A_r^c}, \;\; \mbox{ where  } \;\;
A_r^c = \displaystyle\sum_{j=1}^c A(r,j).
$$
 \end{theorem}
\begin{proof}
If $c=1$, we have $H = \left\langle g_1^{n_1}, \ldots, g_r^{n_r} \right\rangle$  and $[G:H] = n_1 n_2 \cdots n_r = (n_1 n_2 \cdots n_r)^{A_1}$. Now suppose that the result is true for free nilpotent groups of rank $r$ and nilpotency class less than $c$. Using the induction hypothesis on $\frac{G}{\gamma_c(G)}$, we have that 
$[G: H\gamma_c(G)] = (n_1 n_2 \cdots n_r)^{A_r^{c-1}}$. Let us calculate $[\gamma_c(G): \gamma_c(H)]$. A basis for $\gamma_c(G)$ is formed by all basic commutators of weight $c$. Such basis can be written as $X_G = \{ c_i \; | \; i = 1, \ldots, m \}$, where $m = M_r(c)$. A basis for $\gamma_c(H)$ is 
$$
X_H = \left\{ 
c_i^{
     n_1^{\omega_1(c_i)} n_2^{\omega_2(c_i)} \cdots n_r^{\omega_r(c_i)}
    } \;\; \left| \;\; \right.
i = 1, \ldots, m 
\right\}.
$$
In this way,
$$
\begin{array}{ll}
[\gamma_c(G): \gamma_c(H)] & =
\displaystyle\prod_{i=1}^m  n_1^{\omega_1(c_i)} n_2^{\omega_2(c_i)} \cdots n_r^{\omega_r(c_i)}  \\ \\
& = 
 n_1^{A(r,c)} n_2^{A(r,c)} \cdots n_r^{A(r,c)} = (n_1n_2 \cdots n_r)^{A(r,c)}.
\end{array}
$$
Now, since $H$ is a subgroup of finite index in $G$, we have that $\gamma_c(H) = H \cap \gamma_c(G)$. It follows that \\
$$
\begin{array}{ll}
[G:H] &= [G:H\gamma_c(G)][\gamma_c(G): \gamma_c(H)]  
 = (n_1 \cdots n_r)^{A_r^{c-1}} (n_1 \cdots n_r)^{A(r,c)} = \\ \\
&= (n_1 \cdots n_r)^{A_r^c}.
\end{array}$$ 
\end{proof}


 An explicit formula for $A_r^c$ is as follows. 
 
 \begin{lem}
 For $r\geq 2$, an explicit formula for $A_r^c$ is:
$$
\displaystyle{
 A_r^c = \sum_{n=1}^c A(r,n) = \sum_{d=1}^c \mu(d) \left(\dfrac{r^{\left[ \frac{c}{d}\right]}-1}{r-1}   \right),
 }
 $$
 where $\left[ \dfrac{c}{d}\right]$ is the integer part of $\frac{c}{d}$.
 \end{lem}
\begin{proof} We have that
$$
\displaystyle{
\sum_{n=1}^c A(r,n) = \sum_{n=1}^c \dfrac{n}{r} M_r(n) = \dfrac{1}{r} \sum_{n=1}^c \sum_{d|n} \mu(d) r^{\frac{n}{d}}  = \dfrac{1}{r} \sum_{d=1}^c \sum_{\overset{d|n}{1 \leq n \leq c}}  \mu(d) r^{\frac{n}{d}}.
}
$$
Now, observe that $\{n ; d|n \mbox{ and  } 1\leq n \leq c\} = \left\{d,2d, \ldots , \left[ \dfrac{c}{d} \right] d \right\}.$ So we can write
$$
\displaystyle{
\sum_{\overset{d|n}{1 \leq n \leq c}}  \mu(d) r^{\frac{n}{d}} = \sum_{k=1}^{\left[ \frac{c}{d} \right]}  \mu(d) r^k = \mu(d) \sum_{k=1}^{\left[ \frac{c}{d} \right]}  r^k.
}
$$
Finally,
$$
\displaystyle{
 A_r^c = \dfrac{1}{r} \sum_{d=1}^c  \mu(d) \sum_{k=1}^{\left[ \frac{c}{d} \right]}  r^k = 
\sum_{d=1}^c  \mu(d) \sum_{k=1}^{\left[ \frac{c}{d} \right]}  r^{k-1} = 
\sum_{d=1}^c  \mu(d) \left( \dfrac{r^{\left[ \frac{c}{d} \right]}-1}{r-1} \right).}
$$
\end{proof}

Some virtual endomorphisms of the free abelian group $\mathbb{Z}^n$ produce recurrent transitive self-similar representations of $\mathbb{Z}^n$. The extension of these virtual endomorphisms to $\mathbb{R}^n$ tell us that the representation is finite-state if and only if its spectral radius is less than 1 (see \cite{NS}). In \cite{BK}, Bondarenko and Kravchenko extended this result to $\mathfrak{T}$-groups, which will allow us to get information about the automata generation of $N_{r,c}$: 
\begin{theorem}{\bf(Bondarenko, Kravchenko)} \label{BK}
Let $G$ be a $\mathfrak{T}$-group and let $f$ be a simple surjective virtual endomorphism of $G$. Then the transitive self-similar representation induced by $(G,f)$ is finite-state if and only if the spectral radius of $f$ is less than 1.
\end{theorem}

Now we are ready to prove that $N_{r,c}$ is faithfully represented as a recurrent transitive finite-state self-similar group.
\begin{thmA} 
\label{sc}
The free nilpotent group $N_{r,c}$ of class $c$ and finite rank $r$ is recurrent transitive self-similar. Furthermore, $N_{r,c}$ is an automata group. 
\end{thmA} 
\begin{proof}
Consider $G = N_{r,c} =  \left\langle g_1, \ldots, g_r \right\rangle$ and $H = \left\langle g_1^{n_1}z_1, \ldots, g_r^{n_r}z_r \right\rangle \leq G$ with $H \simeq G$, where $n_1, \ldots, n_r$ are positive integers and $z_1, \ldots, z_r \in G'$. We can consider $H = \left\langle g_1^{n_1}, \ldots, g_r^{n_r} \right\rangle$, as the index remains the same, by Theorem \ref{index}. Consider the map $f:G \longrightarrow H$ defined by $g_i^f = g_{i+1}^{n_{i+1}}$ for $i = 1, \ldots, r-1$ and $g_r^f =  g_1^{n_1}$. Then $f$ extends to an epimorphism and therefore $f$ is an isomorphism. Putting $n = n_1 n_2 \cdots n_r$, we have $g_j^{f^{kr}} = g_j^{n^k}$ for all $j = 1,
\ldots, r$ and  $k \geq 1$. Thus  $G^{f^{kr}} \leq G^{n^k}$ for all  $k \geq 1$ and, by Theorem \ref{baumslag},  
$$
\bigcap_{i \geq 1} G^{f^i} \leq
\bigcap_{k \geq 1} G^{f^{kr}} \leq
\bigcap_{k \geq 1} G^{n^k} = 1.
$$
Now, the triple $(G,H,f^{-1})$  provides us a representation $\varphi: G \to Aut(\mathcal{T}_n)$, where $n$ is the index of $H$ in $G$. As ker$\varphi$ is $f^{-1}$-invariant, it follows that ker$\varphi \leq 
($ker$\varphi)^{f^i}$, for all $i \geq 1$. Thus
$$
\mbox{ker}\varphi \leq 
\bigcap_{i \geq 1} (\mbox{ker}\varphi)^{f^i} \leq
\bigcap_{i \geq 1} G^{f^i}  = 1
$$
and the representation is faithful.

Now, lifting the virtual endomorphism $f^{-1}:H \longrightarrow G$, we obtain the matrix
$$
[f^{-1}] =  \left(
\begin{array}[pos]{cccccc}
0&0&\ldots& 0  & \frac{1}{n_1} \\
\frac{1}{n_2}&0& \ldots & 0&0 \\
0& \frac{1}{n_3}& \ldots & 0&0 \\
\vdots& \vdots & \ldots & \vdots & \vdots \\
0 & 0 & \ldots & \frac{1}{n_r} & 0 
	\end{array}
\right).
$$
The characteristic polynomial of $[f^{-1}]$ is $t^r - \dfrac{1}{n_1 n_2 \cdots n_r}$. As $\left|t^r \right| = \dfrac{1}{n_1 n_2 \cdots n_r}< 1$, follows that $|t|<1$ and thus the spectral radius of $[f^{-1}]$ is less than 1. So, by Theorem \ref{BK}, the representation is finite-state.
\end{proof}

\begin{ex} 
Let $G = N_{2, r} = \left\langle g_1, g_2 \ldots, g_r \right\rangle$ and $H=\left\langle g_1^2, g_2 \ldots, g_r \right\rangle \leq G$ with $H \simeq G$. Then the index of $H$ in $G$ is $[G:H]=2^r$. A transversal of $H$ in $G$ is the set of the elements
$1, g_1$ and all the products of the form
$$
\displaystyle{g_1^i\prod_{2\leq j_1< j_2 < ... < j_{k}\leq n}[g_1,g_{j_1}][g_1,g_{j_2}]...[g_1,g_{j_{k}}]}, 
$$
where $i=0,1$ and $1 \leq k \leq r-1$.

Note that $T$ has 
$$ \begin{array}[pos]{rl}
2+ 2\left( \sum_{j=1}^{r-1} {r-1 \choose j} \right) = & 2\left( \sum_{j=1}^{r-1} {r-1 \choose j} +1\right) 
= 2\left( \sum_{j=0}^{r-1} {r-1 \choose j}\right) =  2^r
\end{array} $$ 
elements.
In the particular case $r=3$, we have $G=N_{2,3}=\langle g_1,g_2,g_3\rangle$, $H=\langle g_1^2,g_2,g_3\rangle$ and 
$$T=\{1,g_1,[g_1,g_2],[g_1,g_3],g_1[g_1,g_2],g_1[g_1,g_3],[g_1,g_2][g_1,g_3],g_1[g_1,g_2][g_1,g_3]\}$$ Define the endomorphism $f:H \to G$ that extends the map 
\begin{eqnarray*}
g_{1}^{2}\mapsto g_{3}, \,\, g_{2}\mapsto g_1, \,\, g_{3}\mapsto g_2.
\end{eqnarray*}
With respect to this data, the transitive self-similar representation of $G$ is 
$$G^\varphi \simeq \langle \alpha, \beta, \gamma \rangle,$$
where 
$$\alpha=(e, \gamma, e,e, \gamma,\gamma, e, \gamma)(12)(35)(46)(78), $$
$$\beta=(\alpha,\alpha,\alpha,\alpha,\alpha[\gamma,\alpha],\alpha,\alpha,\,\alpha\gamma[\gamma,\alpha])(25)(68), $$
$$\gamma= (\beta,\beta,\beta,\beta,\beta,\beta[\gamma,\beta],\beta,\beta[\gamma,\beta])(26)(58). $$
\end{ex}


\section{Free Metabelian groups}

The following proposition extends Theorem 1 of \cite{DS1} and has an analogous proof. For the reader’s convenience we supply
a proof.

\begin{prop} \label{4.1}
Let $G$ be a self-similar metabelian group and let $A$ be an abelian subgroup of $G$ such that 
\begin{itemize}
    \item[(i)] $G' \leqslant A$; 
    \item[(ii)]$C_{A}(g) = 1$ for any $g \in G \setminus A$;
    \item[(iii)] There exists $B \leqslant A$ such that $A = B^{G/A} = \oplus_{q \in G/A} B^{q}$.
\end{itemize}
If $G/A$ is torsion free then $A$ is a torsion group of finite exponent. 
\end{prop}

\begin{proof}
Identify $G/A$ with $Q$. Let $f:H \to G$ be a simple virtual endomorphism where $[G:H] = m$. We will prove the proposition in four steps. Suppose by contradiction that $A^{m} \neq 1$.

\begin{enumerate}
    \item If $A_0=H \cap A$, then $A_0^f \leqslant  A$. \\
    
    Since $[A^m, Q^m]$ is normal in $G$, it follows that $[A^m, Q^m]^f \neq 1$. Thus $1 < [A^m, Q^m]^f \leqslant A_{0}^f \cap A$ and $A_{0}^f \cap A$ is central in $AA_{0}^{f}$. Since $C_{A}(g) = 1$ for any $g \in G \setminus A$, we have $A_{0}^f \leqslant A$. \\  
    
    \item For each non-trivial $q \in Q$  and $x_1 , ..., x_t$, $z_1 , ..., z_l \in  Q$, there exists $k$ integer such that
    $$ q^k\{z_1,...,z_l\}\cap \{x_1,...,x_t\} = \emptyset.$$ \\
    
    It is enough to prove that the set $\{k \in \mathbb{Z} \mid q^kz_j\cap \{x_1,...,x_t\}\} \neq \emptyset$ is finite for each $j=1,...,l$. If it is false, there are $j$ and distinct integers $k_1, k_2$ satisfying $$q^{k_1}z_j=q^{k_2} z_j.$$ Then $q^{k_1-k_2}=1$, but $Q$ is torsion-free and we have a contradiction. \\

    \item  If $q \in Q$ is nontrivial, then $(q^m)^f$ is nontrivial. \\
    
    Let $a \in A$ and suppose by contradiction that $(q^m)^f$ is trivial. Then
    
    $$(a^{-m}a^{mq^m})^f = (a^{-m})^f(a^{m})^{f(q^m)^f} = 1.$$
    Thus $A^{m(q^m-1)}$ is a normal subgroup of $G$ contained in the kernel of $f$, a contradiction. \\
    
    \item The subgroup $A^m$ is $f$-invariant. \\

    Since $[G:H]=m$, $A_0$ has finite index in $A$. Consider a transversal $T=\{c_1,...,c_r\}$ of $A_0$ in $A$ and fix $a \in A$.
    
    Since $c_i^m \in A_0$ and $a^m \in A_0$ there exist $x_1,...,x_t,z_1,...,z_l \in Q$  such that 
    $$\langle (c_i^m)^f| i=1,...,r\rangle \,\,\, \leqslant B^{x_{1}} \oplus ... \oplus B^{x_{t}} \text{ and } \langle (a^m)^f\rangle \leqslant B^{z_{1}} \oplus ... \oplus B^{z_{l}}.$$
    
    For each $k\in \mathbb{Z}$, define $i_k\in \{1,...,r\}$ such that $a^{x^{mk}}c_{i_k}^{-1} \in A_0$ (it is possible because $T$ is a transversal of $A_0$ in $A$). Now, $$\left((a^{q^{mk}}c_{i_k}^{-1})^m\right)^f = \left((a^{q^{mk}}c_{i_k}^{-1})^f\right)^m  \in A.$$
    The last equality follows from step 1. because $a^{q^{mk'}}c_{i_k}^{-1} \in A_0$. 
    
   But $A$ is abelian and so $$(a^{q^{mk}}c_{i_k}^{-1})^m=a^{mq^{mk}}c_{i_k}^{-m}.$$ Thus,
   
   $$\left((a^{q^{mk}}c_{i_k}^{-1})^m\right)^f = \left(a^{mq^{mk}}\right)^f\left(c_{i_k}^{-m}\right)^f=\left(a^{m}\right)^{{f}z^k}\left(c_{i_k}^{-m}\right)^f,$$
   
   where $z=(q^m)^f$, which is non trivial by step 3. 
   
   By step 2, we have that there is $k'$ such that 
   $$\{z^{k'}z_1,...,z^{k'}z_l\} \cap \{x_1,...,x_t\} = \emptyset,$$ 
  thus,
   
   $$S := B^{x_{1}} \oplus ... \oplus B^{x_{t}} \oplus B^{z_{1} z^{k'}} \oplus ... \oplus B^{z_{l} z^{k'}}.$$
   
   Then 
   $\left(a^{m}\right)^{{f}z^{k'}}\left(c_{ik}^{-m}\right)^f \in A^m \cap S$, and 
   $$\left(a^{m}\right)^{{f}z^{k'}} \in \bigoplus_{u=1}^{l}B^{mz^{k'}z_u} \leqslant A^m $$
   and we conclude that $(a^m)^f \in A^m$.
\end{enumerate}

Therefore $A^{m} = 1$, a final contradiction.

\end{proof}

\begin{thmB}
The free metabelian group $\mathbb{M}_{r}$ of rank $r \geq 2$ is not transitive self-similar.
\end{thmB}

\begin{proof}
Let $\mathcal{B} = \langle a_{1}, ..., a_{r} \rangle$ and $\mathcal{Q} = \langle q_{1}, ..., q_{r} \rangle$ be two free abelian groups of rank $r$. Then $\mathcal{B} \wr \mathcal{Q} \simeq \mathbb{Z}^r \wr \mathbb{Z}^r$ and by Magnus embedding of wreath products into $2 \times 2$ matrices $\mathbb{M}_{r} \simeq G = \langle a_{1}q_{1}, ..., a_{r}q_{r} \rangle \leq \mathcal{B} \wr \mathcal{Q}$, according \cite{ReSo}.

Let  $f: H \rightarrow G \leq \mathcal{B} \wr \mathcal{Q}$ be a simple virtual endomorphism where $[G : H] = m$. Let $A = G'$, $Q = G / A$, $A_{0} = A \cap H$ and  $T=\{c_1, ..., c_{m_0}\}$ a transversal of $A_0$ in $A$, with $m_{0} = [A:A_{0}]|m$. Since $G$ satisfies conditions $(i)$ and $(ii)$ of Proposition \ref{4.1}, the steps $1$, $2$, and $3$ of its proof follow.

Since $A = \langle [a_iq_i, a_jq_j] \mid i, j = 1, ..., r \rangle^G$ and
$$[a_iq_i, a_jq_j]^{bq} = ([a_i, q_{j}][q_{i}, a_{j}])^{bq} = [a_i, q_{j}]^{q}[q_{i}, a_{j}]^{q}$$
for any $b \in \mathcal{B}^{\mathcal{Q}}$ and any $q \in \mathcal{Q}$, follows that $A$ is a normal subgroup of $\mathcal{B} \wr \mathcal{Q}$ and $A^Q = A^{G \setminus A} = A^{\mathcal{Q}}$. Since $c_i^m \in A_0$ and $[a_i q_i, a_j q_j]^m \in A_0$ there exist $x_1,...,x_t,z_1,...,z_l \in Q$  such that 
    $$\langle (c_i^m)^f| i=1,...,r\rangle \,\,\, \leqslant \mathcal{B}^{x_{1}} \oplus ... \oplus \mathcal{B}^{x_{t}} \text{ and } \langle ([a_i q_i, a_j q_j]^m)^f\rangle \leqslant \mathcal{B}^{z_{1}} \oplus ... \oplus \mathcal{B}^{z_{l}}.$$

As in the proof of step $4$ of Proposition \ref{4.1}, there exist $q \in \mathcal{Q}$ and $k, k' \in \mathbb{Z}$ such that $\left([a_iq_i, a_jq_j]^{m}\right)^{fq^{k'}}\left(c_{ik}^{-m}\right)^f \in A^m \cap S $, where
$$S := \mathcal{B}^{x_{1}} \oplus ... \oplus \mathcal{B}^{x_{t}} \oplus \mathcal{B}^{z_{1} q^{k'}} \oplus ... \oplus \mathcal{B}^{z_{l} q^{k'}} \text{ and } \{q^{k'}z_1,...,q^{k'}z_l\} \cap \{x_1,...,x_t\} = \emptyset.$$
Thus
$$\left([a_iq_i, a_jq_j]^{m}\right)^{{f}q^{k'}} \in \bigoplus_{u=1}^{l}\mathcal{B}^{mq^{k'}z_u} \cap A^m.$$
and $A^m$ is $f$-invariant. Therefore $\mathbb{M}_{r}$ is not transitive self-similar.
\end{proof}


\begin{thebibliography}{99}



\bibitem{BaSu} L. Bartholdi and Z. Sunik, \emph{Some solvable automata groups}, Contemp. Math., \textbf{394} (2006), 11-30.

\bibitem{Baumslag}
G. Baumslag, \emph{``Lecture Notes on Nilpotent Groups"}, Conference Board of the Mathematical Sciences, Regional Conference Series 2, AMS, Providence, Rodhe Island, 1971.

\bibitem{AS}
A. Berlatto and S. Sidki, \emph{Virtual Endomorphisms of Nilpotent Groups}. Groups, Geometry, and Dynamics, {\bf 1}
(2007) 21-46.

\bibitem{BK}
I. V. Bondarenko, R.V. Kravchenko, \emph{Finite-state self-similar actions of nilpotent groups}. Geometriae Dedicata, 163, 339-348, 2013.

\bibitem{BS} A. M. Brunner and S. N. Sidki, \emph{Abelian state-closed subgroups
of automorphisms of $m$-ary trees}. Groups, Geometry, and Dynamics, 
\textbf{4} (2010), 455 - 471.

\bibitem{DS1} A. C. Dantas and S. N. Sidki, On self-similarity of wreath
products of abelian groups, Groups, Geometry
and Dynamics, \textbf{12} (2018), 1061–1068.



\bibitem{mhall}
M. Hall Jr., \emph{``The Theory of Groups"}, 2nd. edition, Chelsea Publishing Company, New York, 1976.

\bibitem{phall}
P. Hall, \emph{``Nilpotent Groups"}, Queen Mary College Mathematical Notes, London  1969.

\bibitem{K} M. Kapovich, \emph{Arithmetic aspects of self-similar groups}. Groups,
Geometry, and Dynamics {\bf 6} (2012), 737-754.



\bibitem{NS}
V. Nekrashevych and S. Sidki, \emph{Automorphisms of the binary tree: state-closed subgroups and dynamics of $1/2$-endomorphisms}. In T.W.Muller, editor, Groups: Topological, Combinatorial and Arithmetical Aspects, volume 311 of LMS Lecture Note Series, 375-404, 2004.

\bibitem{ReSo} 
 V. N. Remeslennikov and V.G. Sokolov, {\it Certain properties of the Magnus embeddings}. Algebra i Logika.  \textbf{9} (1970) 566–578.


\bibitem{SaS} D. M. Savchuk and S. N. Sidki, \textit{Affine automorphisms of
rooted trees,} Geometriae Dedicata, \textbf{183} (2016), 195-213.


\bibitem{smith}
G. C. Smith, \emph{Compressibility in Nilpotent Groups}, {\em Bull. London Math. Soc.}  \textbf{17} (1985), 453-457.

\end{thebibliography}
\end{document}